\documentclass[11pt]{amsart}
\usepackage[utf8]{inputenc}
\usepackage{mathtools}
    \mathtoolsset{centercolon}
\usepackage{amssymb}
\usepackage{enumitem}
\usepackage{thmtools}
\usepackage{hyperref}
\usepackage{scalerel}
\usepackage{tikz}
\usepackage{tikz-cd}

\theoremstyle{plain}
  \declaretheorem[within=section]{theorem}
  \declaretheorem[name={Theorem}]{introtheorem}
    
  \declaretheorem[name={Theorem}]{maintheorem}
    
  \declaretheorem[numbered=no,name={Theorem}]{theorem*}
  \declaretheorem[sibling=theorem]{proposition}
  \declaretheorem[sibling=theorem]{lemma}
  \declaretheorem[sibling=theorem]{corollary}
\theoremstyle{definition}
  \declaretheorem[sibling=theorem]{remark}
  \declaretheorem[sibling=theorem]{definition}
  \declaretheorem[sibling=theorem,name={Proposition-Definition}]{propdef}
  \declaretheorem[sibling=theorem,qed={$\diamondsuit$}]{example}

\newcommand{\C}{\mathbf{C}}
\newcommand{\K}{\mathbf{K}}
\newcommand{\Q}{\mathbf{Q}}
\newcommand{\R}{\mathbf{R}}
\newcommand{\Si}{\mathbf{S}}
\newcommand{\T}{\mathbf{T}}
\newcommand{\OTR}{\mathcal{O}_{\TR}}
\newcommand{\ORR}{\mathcal{O}_{\R[[t^\R]]}}
\newcommand{\TR}{ % Tropical Hyperfield
    {\mathbf{T}\!\!\mathbf{R}}
}
\DeclareMathOperator{\mult}{mult}
\DeclareMathOperator{\sign}{sign}
\DeclareMathOperator{\newt}{Newt}

\DeclareMathOperator*{\bigboxplus}{\raisebox{-0.5em}{\scaleobj{2}{\boxplus}}}
\DeclareMathOperator{\fr}{Frac}
\DeclareMathOperator{\In}{in}
\DeclareMathOperator{\supp}{supp}

\usepackage[
    backend=biber,
    style=alphabetic,
    url=false,
    doi=true,
    eprint=true
]{biblatex}
\usepackage{hyperref}
\hypersetup{colorlinks=true}

\title[Multiplicities over the Signed Tropical Hyperfield]{A Newton Polygon Rule for Formally-Real Valued Fields and Multiplicities over the Signed Tropical Hyperfield}
\author{Trevor Gunn}
\email{\href{mailto:tgunn@gatech.edu}{tgunn@gatech.edu}}
\address{School of Mathematics, Georgia Institute of Technology, Atlanta, USA}

\begin{document}
\begin{abstract}
By defining multiplicities for zeros of polynomials over hyperfields, Baker and Lorscheid were able to provide a unifying perspective on Descartes's rule and the Newton polygon rule for polynomials over a formally-real and valued field respectively. In this paper, we apply their multiplicity formula to the hyperfield associated with formally-real, valued fields to prove a Newton polygon rule which combines Descartes's rule of signs with the classical Newton polygon rule.

\bigskip


\bigskip

\end{abstract}
\maketitle

\section{Introduction}
Hyperfields are a generalization of fields whose addition is set-valued. The notion of multi-valued groups goes back to at least Frédéric Marty in the 1930's and hyperrings and hyperfields to Marc Krasner in the 1950's.
In the context of valued fields and tropical geometry, Oleg Viro described how tropical varieties arise as equations over hyperfields, replacing the standard ``minimum occurs twice'' formulation with an algebraic formulation \cite{V1, V2}. We direct the reader to Viro's paper \cite{V1} for more on the history of multi-valued algebras.

Recently, Matt Baker and Nathan Bowler have used hyperfields as a unifying framework to study matroids, oriented matroids and valuated matroids~\cite{BB}.
And then later, Matt Baker and Oliver Lorscheid have shown how Descartes's rule of signs and Newton's polygon rule can both be obtained from a single definition of multiplicity over hyperfields \cite{BL}.

Baker and Lorscheid (\emph{loc. cit.}) define multiplicities of roots of polynomials over a hyperfield $H$ via the recursive formula
\[ \mult_a(p) \coloneqq \mult^H_a(p) \coloneqq \begin{cases}
    0 & \text{if } p(a) \not\ni 0, \\
    1 + \max \{\mult^H_a(q) : p \in (x - a) q\} & \text{if } p(a) \ni 0.
\end{cases} \]
Baker and Lorscheid prove the following two theorems.

\begin{theorem*}[Baker, Lorscheid]
For the hyperfield of signs, $\mult^{\Si}_{+1}(p) = \Delta(p)$, where $\Delta(p)$ is the number of sign changes in the coefficients of $p$.
\end{theorem*}

This multiplicity is equal (modulo pairs of complex roots) to the number of positive roots of any lifting of $p$ to a real-closed field.

\begin{theorem*}[Baker, Lorscheid]
For the tropical hyperfield $\T$, \[ \mult^\T_{a}(p) = \mult^\K_1(\In_{-a}(p)) = n - m \] if $\In_{-a}(p) = \sum_{i = m}^n c_ix^i$ with $c_mc_n \ne 0$. Here $\K = \{0, \infty\} \subseteq \T$ is the Krasner hyperfield---the hyperfield analogue of the Boolean semiring.
\end{theorem*}

This multiplicity is equal to the number of roots of $P$ with valuation $a$ for any lifting $P$ of $p$ to an algebraically closed, non-Archimedean valued field.

In this paper, we prove a similar result for the signed tropical hyperfield, $\TR$ which combines features of both the tropical and signed hyperfields.

\begin{introtheorem}
For an element $a = (+1, r)$ of the signed tropical hyperfield,
\[ \mult^{\TR}_{a}(p) = \mult^\Si_{+1}(\In_{-r}(p)) = \Delta(\In_{-r}(p)).\]
\end{introtheorem}

This multiplicity is an upper bound on the number of roots with valuation $r$ and sign $+1$ of any lifting of $p$ to a real-closed, non-Archimedean-ordered valued field. The gap between this upper bound and the number of such roots is an even number corresponding to pairs of complex-conjugate roots.

\begin{remark}
If the sign of $a$ is negative ($a = (-1, r)$) then the multiplicity is
\[ \mult_a^\TR (p) = \mult_{-a}^\TR (p(-x)) = \mult_{-1}^\Si (\In_{-r}(p)) = \Delta(\In_{-r}p(-x)). \]

Additionally, for any hyperfield $H$,
\[ \mult_0^H \left( \sum c_i x^i \right) = \min \{i : c_i \neq 0\}. \]
Thus the theorems of Baker-Lorscheid and Theorem~A completely classify the multiplicities over $\Si, \T, \TR$ respectively.
\end{remark}

We also give a proof of a related theorem in the language of polynomials over a real-closed, non-Archimedean-ordered valued field.

\begin{introtheorem}
    Let $K$ be a real-closed, non-Archimedean-ordered valued field. Let $P(x)$ be a polynomial over $K$. Then the number of real roots of $P$ which are positive and have valuation $r$ is bounded above by and congruent mod $2$ to $\Delta(\In_{-r}(P))$.
\end{introtheorem}

We can think of $\Delta(\In_{-r}(P))$ as counting the number of sign changes along the edge of slope $-r$ in the Newton polygon of $P$ as explained in the following example.

\begin{example} \label{ex:first}
    Let us consider the field $\R((y))$ of Laurent series in $y$.
    Define a valuation $v$ on $\R((y))$ by
    \[ v\left( \sum_i a_iy^i \right) = \min\{i : a_i \neq 0\}. \]

    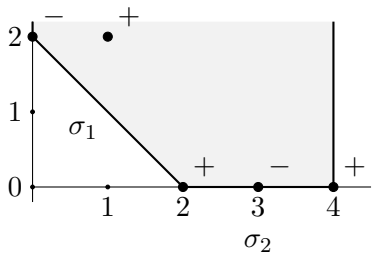
\begin{figure}[htbp]
        \centering
        \begin{tikzpicture}
            \filldraw [color=black!5,draw=black,thick] (0,2.2) -- (0,2) -- (2,0) -- (4,0) -- (4,2.2);

            \foreach \x/\y/\s in {0/2/-,1/2/+,2/0/+,3/0/-,4/0/+}
                \filldraw (\x,\y) circle (0.6mm) node [above right] {$\s$};

            \draw (0,-0.2) -- (0,2.2);
            \foreach \y in {0,1,2}
                \filldraw (0,\y) circle (0.25mm) node [left] {$\y$};

            \draw (0,0) -- (4.5,0);
            \foreach \x in {1,2,3,4}
                \filldraw (\x,0) circle (0.25mm) node [below] {$\x$};

            \draw (1,1) node [below left] {$\sigma_1$};
            \draw (3,-0.5) node [below] {$\sigma_2$};
        \end{tikzpicture}
        \caption{Newton polygon of Example~\ref{ex:first} with the signs of the associated coefficients indicated.}
        \label{fig:classical-example}
    \end{figure}

    Let us consider the Newton polygon of the following polynomial over $R((y))$:
    \[ P(x) = (x - y)(x + y)(x - 1)^2 = -y^2 + 2y^2x + (1 - y^2)x^2 - 2x^3 + x^4. \]
    Let $c_i(y)$ be the coefficient of $x^i$. We plot the points $(i, v(c_i))$, and take the lower convex hull to form the Newton polygon. See Figure~\ref{fig:classical-example}. We will also write the sign of the leading term of $c_i$ at the associated point.

    Theorem~B tells us that the number of roots of $P$ whose leading term is linear and whose leading coefficient is positive is bounded above, and congruent modulo $2$ to the number of sign changes inside the edge $\sigma_1$ from $(0,2)$ to $(2,0)$. Note that the sign at $(1,2)$ is irrelevant because that point is not inside $\sigma_1$. There is 1 sign change inside $\sigma_1$ and it corresponds to the root $x = y$.
    
    Additionally, the number of roots whose leading term is constant, and whose leading coefficient is positive is bounded above by the number of sign changes inside the edge $\sigma_2$ from $(2,0)$ to $(4,0)$. There are 2 sign changes here, corresponding to the double root at $x = 1$.

    Finally, if we look at the signs of $P(-x)$ then $\sigma_1$ still has one sign change, corresponding to the root $x = -y$, but $\sigma_2$ now has $0$.
\end{example}

We can think of Theorem~A as providing the maximal multiplicity over all possible lifts. We will see that this maximum is achieved for some lift in Theorem~\ref{thm:lifting} although the lift constructed there only guarantees that the leading terms of roots in the lift are real.

\subsection{Acknowledgements}
The author thanks Matt Baker, Marcel Celaya, Tim Duff, Oliver Lorscheid and Josephine Yu for their helpful discussions regarding the content of this paper.

In particular, the author thanks Matt Baker for suggesting what the Newton polygon rule should look like, for his encouragement on completing this project and for his comments on an earlier draft of this paper. The author thanks Marcel Celaya for pointing out how to use initial forms in the proof of Theorem~\ref{thm:classical}. The author thanks Yoav Len for helpful feedback on the first draft of this paper. And finally, the author thanks Oliver Lorscheid for suggesting the map $\cdot |_{t = 0} \colon \OTR \to \Si$ defined in Section~\ref{sec:morphisms} that is needed to prove claim (1) of the Theorem~\ref{thm:main}.

\section{Preliminaries}
\subsection{Real fields}
A \emph{real field} (or \emph{formally-real field}) is a field in which $-1$ is not a sum of squares. Every real field admits a not-necessarily-unique ordering that makes it into an ordered field.

A real field which admits no proper, real algebraic extensions is called \emph{real-closed}. Every real-closed field has a unique ordering that makes it an ordered field and that ordering is defined by $x \ge 0$ if $x = y^2$ for some $y$. By a theorem of Artin and Schreier, $K$ is real-closed if and only if $K$ is not algebraically closed and the algebraic closure of $K$ is $K[\sqrt{-1}]$.

If $K$ is a real-closed field, its \emph{sign function} is defined by $\sign(x) = +1$ if $x > 0$, $\sign(x) = 0$ if $x = 0$ and $\sign(x) = -1$ if $x < 0$. Additionally, $K$ has an absolute value given by $|x| = \sign(x) x$.

\subsection{Valued fields} \label{sec:valuedfields}
A \emph{rank 1, non-Archimedean valued field} $K$ is a field together with a map $v \colon K \to \R \cup \{\infty\}$ such that $v(0) = \infty$ and such that $v \colon (K^\times, \cdot) \to (\R, +)$ is a group homomorphism that satisfies the relation
\[ v(x + y) \ge \min\{v(x), v(y)\} \]
for all $x,y \in K$. For simplicity, we will refer to these as just ``valued fields.'' We call the map $v$ a \emph{(non-Archimedean) valuation}.

The valuation ring of $K$ is the ring $\mathcal{O}_K \coloneqq \{ x \in K : v(x) \ge 0 \}$. This is a local ring with maximal ideal $\mathfrak{m}_K \coloneqq \{ x \in K : v(x) > 0 \}$. We call $k \coloneqq \mathcal{O}_K/\mathfrak{m}_K$ the residue field of $K$. If $K$ is algebraically closed then so is its residue field.

\subsubsection{Ordered valued fields}
An ordering on a valued field is called \emph{Arch\-i\-me\-de\-an} if for every $x \in K$, there is a natural number $n$ such that $x \le n$. A \emph{non-Archimedean order} is an order which is not Archimedean.

\begin{remark}
When $K$ is real, it is not necessarily true that $k$ is real. For example $\Q$ is real but with a $p$-adic valuation, its residue field is not real. A necessary and sufficient condition for $k$ to be real is that $\mathcal{O}_K$ is convex with respect to the ordering on $K$ \cite[Proposition~2.2.4.]{EP}.
\end{remark}

\begin{definition}
If $\le$ is an ordering on $K$, then the \emph{$\le$-convex-hull} of a subring $R$ of $K$ is the set
\[ \mathcal{O}_R(\le) = \{ x \in K : x \le a \text{ and } -x \le a \text{ for some } a \in R \}. \]
$R$ is \emph{$\le$-convex}---or simply \emph{convex}---if $R = \mathcal{O}_R(\le)$.
\end{definition}

When $\mathcal{O}_K$ is convex, then the ordering on $K$ is non-Archimedean \cite[Corollary~2.2.6.]{EP} and we call $K$ a \emph{non-Archimedean-ordered field}. Going forward, when we say $K$ is a real (or real-closed) valued field, we mean that $K$ is real (or real-closed) and its residue field is also real (or real-closed). We need this hypothesis in order to have a map from $K$ to the signed tropical hyperfield (Example~\ref{ex:TRmaps}).

If $K$ and $k$ are both real-closed, then $K$ is Henselian and the valuation on $K$ extends uniquely to a valuation on the algebraic closure, $K[\sqrt{-1}]$ \cite[Theorem~4.1.3. and Theorem~4.3.7.]{EP}. Additionally, the residue field of $K[\sqrt{-1}]$ is $k[\sqrt{-1}]$. Furthermore, valuations on Henselian fields are Galois-invariant, in particular, two complex-conjugate roots of a polynomial must have the same valuation.

See \cite[Chapter~4.3]{EP} for more on Henselian valued fields and, in particular, for ordered (i.e. formally-real), Henselian valued fields.

\subsubsection{Hahn Series}
There are many examples of valued fields. Here we give one general family of valued fields which is well-suited for our discussion.

Let $k = \R$ or $\C$ (more generally any real or algebraically closed field) and let $\Gamma \subseteq \R$ be an ordered group. We call elements of
\[
    k[[t^\Gamma]] = \left\{ \sum_{\gamma \in I} a_\gamma t^\gamma : a_\gamma \in k, I \subset \Gamma \text{ is well-ordered} \right\}
\]
\emph{Hahn series}. Let us note that every well-ordered subset of $\R$ is countable.

When $k$ is algebraically closed (resp. real-closed), and $\Gamma$ is divisible, then the Hahn series form an algebraically closed (resp. real-closed) field \cite[Theorem~1]{ML}.

The two Hahn series fields we will make use of are $k[[t^\R]]$ with $k = \R$ or $k = \C$. By the above fact, those fields are real-closed and algebraically closed respectively.

There are surjective maps
\[ v_\C \colon \C[[t^\R]] \longrightarrow \R \cup \{ \infty \} \text{ and } v_\R \colon \R [[t^\R]] \longrightarrow (\{\pm 1\} \times \R) \cup \{\infty\} \]
given by $v_k(0) = \infty$ and otherwise
\begin{align*}
    v_\C\left( \sum_{\gamma \in I} c_\gamma t^\gamma \right) &= \gamma_0, \\
    v_\R\left( \sum_{\gamma \in I} c_\gamma t^\gamma \right) &= (\sign(c_{\gamma_0}), \gamma_0),
\end{align*}
where $\gamma_0 = \min I$.

$v_\C$ is a valuation on $\C[[t^\R]]$ that makes it (and $\R[[t^\R]]$) into a non-Archimedean valued field. The field $\R[[t^\R]]$ has a non-Archimedean ordering given by $0 < \sum_{\gamma \in I} c_\gamma t^\gamma$ if $0 < c_{\min I}$.

\subsection{Hyperfields} \label{sec:hyperfields}

The \emph{tropical} and \emph{signed tropical hyperfield} are respectively
\[ \T = (\R \cup \{\infty\}, \cdot, \boxplus, 0, \infty) \text{ and } \TR = ((\{\pm 1\} \times \R) \cup \{ \infty \}, \cdot, \boxplus, (1,0), \infty).\]
One way to define the multiplication, $\cdot$, and the hyperaddition, $\boxplus$, is via the maps $v_k$ on $k[[t^\R]]$.

The multiplication on $\T$ ($k = \C$) and on $\TR$ ($k = \R$) can be described by the rule
\[ a \cdot b = v_k(\alpha \cdot \beta) \text{ if } v_k(\alpha) = a \text{ and } v_k(\beta) = b. \]
The hyperaddition is defined by
\[ a \boxplus b = \{v_k(\alpha + \beta) : v_k(\alpha) = a, v_k(\beta) = b\}. \]
Note: multiplication is well-defined.

\begin{remark} \label{rem:Viro'shyperfield}
    Viro uses a different, but isomorphic, signed tropical hyperfield; he also calls it the \emph{real tropical hyperfield} \cite{V1, V2}. Viro's hyperfield has $\R$ as its underlying set and is related to ours via the isomorphism $(s, r) \mapsto s e^{-r}$ and $\infty \mapsto 0$ ($e = 2.718\dots$).
\end{remark}

Explicitly, the multiplication for $\T$ is given by the usual \textbf{addition} in $\R$. Multiplication for $\TR$ is given by $(s,r) \cdot (s',r') = (ss', r + r')$ (multiply the signs normally, multiply the valuations tropically).

Hyperaddition in $\T$ is given by
\[
    a \boxplus b =
    \begin{cases}
        \{a\} & a < b, \\
        \{b\} & b < a, \\
        [a,\infty] & a = b.
    \end{cases}
\]
Hyperaddition in $\TR$ is given by $(s, r) \boxplus \infty = \infty \boxplus (s,r) = \{ (s,r) \}$ and otherwise
\begin{equation} \tag{$*$} \label{eq:TRaddition}
    (s,r) \boxplus (s',r') =
    \begin{cases}
        \{ (s, r) \} & r < r', \\
        \{ (s', r') \} & r' < r, \\
        \{ (s, r) \} & r = r' \text{ and } s = s', \\
        \{ (\pm 1, t) : t \ge r \} \cup \{\infty\} & r = r' \text{ and } s = - s'.
    \end{cases}
\end{equation}

We define $a_1 \boxplus a_2 \boxplus \cdots \boxplus a_n$ inductively by
\[ a_1 \boxplus a_2 \boxplus \cdots \boxplus a_n = \bigcup_{b \in a_1 \boxplus \cdots \boxplus a_{n-1}} b \boxplus a_n. \]
One can check that this is the same as $\{v_k(\alpha_1 + \dots + \alpha_n) : v_k(\alpha_i) = a_i\}$.
Thus the addition is associative and commutative.

The \emph{Krasner hyperfield}, $\K$, is the set $\{0, 1\}$ with the usual multiplication and where $0 \boxplus x = \{x\}$ and $1 \boxplus 1 = \{0, 1\}$. We identify this with the subset $\{0, \infty\}$ of $\T$ via $0 \mapsto \infty, 1 \mapsto 0$. Note: $\K$ is the hyperfield analogue of the Boolean semiring $(\{0,1\}, \cdot, \max, 1, 0)$ and $\T$ of the tropical semiring $(\R, \cdot, \min, \infty, 0)$.

The \emph{sign hyperfield} (or \emph{hyperfield of signs}) is $\Si = (\{0, 1, -1\}, \cdot, \boxplus, 0, 1)$ with the usual multiplication and hyperaddition given by the following table
\begin{center}
    \begin{tabular}{c|c|c|c}
        $\boxplus$ & $0$ & $-1$ & $+1$ \\\hline
        $0$ & $\{0\}$ & $\{-1\}$ & $\{+1\}$ \\\hline
        $-1$ & $\{-1\}$ & $\{-1\}$ & $\Si$ \\\hline
        $+1$ & $\{+1\}$ & $\Si$ & $\{+1\}$
    \end{tabular}
\end{center}
We can identify $\Si$ with the subset $\{\infty, (1,0), (-1,0)\} \subset \TR$ where $\infty \leftrightarrow 0$ and $(\pm 1, 0) \leftrightarrow \pm 1$ (\emph{cf.} Remark~\ref{rem:Viro'shyperfield}). Thus $\Si$ is a sub-hyperfield of $\TR$.

We define $\OTR \coloneqq (\{\pm 1\} \times \R_{\ge 0}) \cup \{ \infty \} \subset \TR$. This is a sub-hyperring of $\TR$ and it equals the image of $\ORR$ under $v_\R$.

\begin{remark}
The ordering on $\R[[t^\R]]$ is given by $0 < \alpha$ if $\alpha \ne 0$ and the leading coefficient of $\alpha$ is positive. This induces an ordering on $\TR$ such that $\infty < (s,r)$ if $s = +1$.

More generally, $(s,r) < (s', r')$ if either
\begin{enumerate}
    \item $(s', r') \boxplus (-1,0) \cdot (s, r) = \{(+1, r'')\}$---the first three cases of the addition given in \eqref{eq:TRaddition},
    \item or $r = r'$ and $s < s'$.
\end{enumerate}
\end{remark}

\subsection{Axioms for hyperrings} We now list the axioms of hyperrings and hyperfields and indicate where they differ from those of rings and fields. Let us remark that every (commutative) ring is a hyperring and every field is a hyperfield, so the axioms cannot differ too greatly.

We will require these axioms to develop the theory of morphisms between hyperfields and multiplicities over hyperfields in this section and the next.
\begin{definition}
    A \emph{hyperring} $(H, \cdot, \boxplus, 1, 0)$ consists of
    \begin{itemize}[leftmargin=1em]
        \item An Abelian \emph{monoid-with-zero} $(H, \cdot, 1, 0)$ satisfying
        \begin{enumerate}[label=(M\arabic*),itemindent=5mm]
            \item $\cdot$ is associative and commutative,
            \item $1 \cdot x = x$ for all $x$,
            \item $0 \cdot x = 0$ for all $x$.
        \end{enumerate}
        These are the same axioms $(R,\cdot,1,0)$ satisfies if $R$ is a commutative ring.
        \item A binary \emph{hyperoperation} $\boxplus \colon H \times H \to \mathcal{P}(H) \setminus \{\varnothing\}$ which has the following axioms:
        \begin{enumerate}[label=(H\arabic*),itemindent=5mm]
            \item $\boxplus$ is commutative and associative, meaning
            \[ a \boxplus b = b \boxplus a \text{ and } \bigcup_{d \in b \boxplus c} a \boxplus d = \bigcup_{d \in a \boxplus b} d \boxplus c, \]
            \item $0 \boxplus a = \{a\}$ for all $a$,
            \item for every $a$ there exists a unique $-a$ such that $a \boxplus (-a) \ni 0$, \label{H:neg}
            \item $a \cdot (b \boxplus c) \coloneqq \{ ad : d \in b \boxplus c \} = ab \boxplus ac$,
            \item $a \in b \boxplus c$ if and only if $(-b) \in a \boxplus c$. \label{H:reverse}
        \end{enumerate}
        Axiom \ref{H:reverse} exists in lieu of a proper subtraction operation. Otherwise these axioms exactly mirror those of a commutative ring.
    \end{itemize}
\end{definition}

\begin{definition}
    A hyperring $(H, \cdot, \boxplus, 1, 0)$ is called a \emph{hyperfield} if additionally $0 \ne 1$ and every non-zero element of $H$ has a multiplicative inverse, in which case $({H \setminus \{0\}},\cdot,1)$ is an Abelian group.
\end{definition}

\begin{remark} \label{rem:abc}
    \ref{H:neg} implies that $a \in b \boxplus c$ if and only if $0 \in (-a) \boxplus b \boxplus c$ (because the inverse is unique).
\end{remark}

We leave the following lemma as an easy exercise.

\begin{lemma}
    If $R$ is any commutative ring and $G \subseteq R^\times$ is any multiplicative subgroup, then the set of cosets $H = R/G$ is a hyperring with the usual multiplication, $0 = [0]$ and $1 = [1]$ and whose hyperaddition is defined by
    \[ [a] \boxplus [b] = \{ [a' + b'] : a' \in [a], b' \in [b]\}. \]
\end{lemma}

\begin{remark}
    This lemma implies that any field is a hyperfield ($G = \{1\}$) and that the hyperrings in Section~\ref{sec:hyperfields} satisfy the above axioms. Here
    \begin{itemize}
        \item $\K \cong k/k^\times$ for any field other than the field with 2 elements,
        \item $\Si \cong \R/\R_{> 0}$,
        \item $\T \cong \C[[t^\R]]/v_\C^{-1}(0)$,
        \item $\TR \cong \R[[t^\R]]/v_\R^{-1}((+1,0))$,
        \item $\OTR \cong \ORR/v_\R^{-1}((+1,0))$.
    \end{itemize}
\end{remark}

\subsection{Morphisms of hyperfields} \label{sec:morphisms}
A morphism $f\colon H_1 \to H_2$ between two hyperfields is a function satisfying $f(0) = 0, f(1) = 1$, $f(ab) = f(a)f(b)$ and $f(a \boxplus b) \subseteq f(a) \boxplus f(b)$.

In the following examples and proposition, let $R$ be any commutative ring.

\begin{example}
Maps $v \colon R \to \K$ correspond to prime ideals of $R$ where the prime ideal is $\mathfrak{p} = v^{-1}(0)$.
\end{example}

\begin{example}
Maps $v \colon R \to \Si$ correspond to a prime ideal $\mathfrak{p} = v^{-1}(0)$ and an ordering on $R/\mathfrak{p}$ given by $0 < x$ if $v(x) = +1$.
\end{example}

\begin{example}
Maps $v \colon R \to \T$ correspond to a prime ideal $\mathfrak{p} = v^{-1}(\infty)$ and a rank-1 valuation on $\fr(R/\mathfrak{p})$.
\end{example}

\begin{example} \label{ex:TRmaps}
Maps $v \colon R \to \TR$ correspond to
\begin{itemize}
    \item a prime ideal $\mathfrak{p} = v^{-1}(\infty)$;
    \item a rank-1 valuation on $\fr(R/\mathfrak{p})$;
    \item an non-Archimedean-ordering on $\fr(R/\mathfrak{p})$. \qedhere
\end{itemize}
\end{example}

The key morphisms that will be used in what is to follow are the arrows of the following diagram.

\begin{center}
    \begin{tikzcd}
        \R \arrow[r, shift left] \arrow[d, "\sign"'] &
            \ORR \arrow[r] \arrow[l, shift left, "\cdot |_{t = 0}"] \arrow[d, "v_\R"'] &
            \R[[t^\R]] \arrow[r] \arrow[d, "v_\R"'] &
            \C[[t^\R]] \arrow[d, "v_\C"'] \\
        \Si \arrow[r, shift left] &
            \OTR \arrow[r] \arrow[l, shift left, "\cdot |_{t = 0}"] &
            \TR \arrow[r, "|\cdot|"] \arrow[ll, bend left=40, "\sign"] &
            \T
    \end{tikzcd}
\end{center}
The unnamed arrows are inclusions. The map $\sign \colon \R \to \Si$ is the usual sign function. The maps $\sign \colon \TR \to \Si$ and $| \cdot | \colon \TR \to \T$ are the projection onto the first and second factor respectively, with $\sign(\infty) = 0$ and $|\infty| = \infty$.

The map $\cdot|_{t = 0}$ is defined for $\ORR$ as evaluation at $t = 0$ and for $\OTR$ as $(s,0)|_{t=0} = s$ and $(s,r)|_{t=0} = 0$ if $r > 0$.

\section{Polynomials and Multiplicities over a Hyperfield}
A polynomial over a hyperfield $H_1$ is a formal sum $p(x) = \sum_{i = 0}^n c_ix^i$ with $c_i \in H_1$. If $f\colon H_1 \to H_2$ is a morphism of hyperfields, then $f(p) \coloneqq \sum_{i = 0}^n f(c_i)x^i$.

We now show that factorization of polynomials is preserved under morphisms.

\begin{lemma}\label{lem:morphism-factorization}
    If $\sum a_ix^i \in \left( \sum b_jx^j \right)\left( \sum c_kx^k \right)$---which we take to mean that $a_i \in \smash{\bigboxplus\limits_{i = j + k}} b_jc_k$ for all $i$---and $f$ is any morphism of hyperfields, then \[ \sum f(a_i)x^i \in \left( \sum f(b_j)x^j \right)\left( \sum f(c_k)x^k \right). \]
\end{lemma}
\begin{proof}
    By Remark~\ref{rem:abc}, $a \in b \boxplus c$ implies $0 \in (-a) \boxplus b \boxplus c$ implies $0 \in (-f(a)) \boxplus f(b) \boxplus f(c)$ implies $f(a) \in f(b) \boxplus f(c)$.

    By induction on the number of summands, $a_i \in \bigboxplus\limits_{i = j + k} b_jc_k$ implies
    \[ f(a_i) \in \bigboxplus\limits_{i = j + k} f(b_j)f(c_k) \]
    as required.
\end{proof}

\subsection{Multiplicities}

Baker and Lorscheid define zeroes and multiplicities over a hyperfield as follows \cite[Lemma~A, Definition~1.4]{BL}.

\begin{propdef}
    Let $p$ be a polynomial over a hyperfield $H$. The following are equivalent:
    \begin{enumerate}
        \item $0 \in \bigboxplus\limits_{i = 0}^n c_ia^i$
        \item There exist elements $d_0,\dots,d_{n-1} \in H$ such that
        \[ c_0 = -ad_0, c_i \in (-ad_i) \boxplus d_{i-1} \text{ for } i \in \{1,\dots,n-1\}, c_n = d_{n-1}. \]
    \end{enumerate}
    Baker and Lorscheid (after Viro) write $0 \in p(a)$ if (1) is satisfied. The classification (2) is new to Baker and Lorscheid's work and they abbreviate it as $p \in (x - a)q$ where $q = \sum d_ix^i$---this is compatible with the notation introduced in Lemma~\ref{lem:morphism-factorization}. We say that $a$ is a root (or zero) of $p$ if these conditions hold.
\end{propdef}

\begin{definition}
    For a polynomial $p(x) = \sum c_ix^i \in H[x]$ and $a \in H$, Baker and Lorscheid define $\mult_a(p)$ recursively by
    \[ \mult_a(p) =
    \begin{cases}
        0 & p(a) \not\ni 0 \\
        1 + \max\{\mult_a(q) : p \in (x - a) q\} & p(a) \ni 0.
    \end{cases}
    \]
    We will also use the notation $\mult_a^H(p)$ if we need to indicate that $p \in H[x]$.
\end{definition}

There are several pathologies to multiplicities over hyperfields. For instance, a polynomial can have infinitely many roots and even when there are finitely many roots, the sum of the multiplicities might exceed the degree of the polynomial.
Additionally, the factorization $p \in (x - a)q$ is generally not unique and $\mult_a(q)$ is generally not independent of $q$. We point the reader to \cite{BL} for these examples and further details.

\subsection{Newton Polygons}

\begin{definition}
    If $p = \sum_{i = 0}^n c_ix^i$ is a polynomial over a valued field $(K,v)$, its Newton polygon, $\newt(p)$, is the lower convex hull of $\{ (i, v(c_i)) : i = 0,\dots,n \}$. That is, it is the set defined by the lower inequalities of the convex hull where a lower inequality is an inequality $\langle u, x \rangle \ge 0$ with $u$ in the upper half-plane. See Figure~\ref{fig:newton-poly-example} for a picture.
\end{definition}

The definition of $\newt(p)$ only depends on the valuation of the coefficients, therefore it makes sense to define $\newt(p)$ the same way when $p \in \T[x]$. That is, $\newt(p)$ is the lower convex hull of $\{ (i, c_i) : i = 0,\dots,n \}$ when $p = \sum c_ix^i \in \T[x]$. When $p = \sum c_ix^i \in \TR[x]$, we define $\newt(p) = \newt(|p|)$ where $|p| \coloneqq \sum |c_i|x^i \in \T[x]$.

An \emph{edge} of $\newt(p)$ will always mean a bounded edge. The \emph{horizontal length} of such an edge is the length of its projection onto the $x$-axis.

\begin{example}
    Consider $p = (x - t^a)^n \in \C[[t^\R]][x]$. The valuation of the $i$-th coefficient is
    \[ v_\C\left( \binom{n}{i}t^{(n-i)a} \right) = (n-i)a. \]
    Therefore, the Newton polygon of $p$ is the lower convex hull of $\{(i, (n - i)a) : 0 \le i \le n\}$. This has a single edge of slope $-a$ and horizontal length $n$.
\end{example}

\begin{example}
    Let $p = 2 + \infty x + 1x^2 + (-\tfrac12) x^3 + (-1)x^4 + \tfrac12x^5 + 1x^6 \in \T[x]$. The Newton polygon of $p$ is pictured in Figure~\ref{fig:newton-poly-example}.

    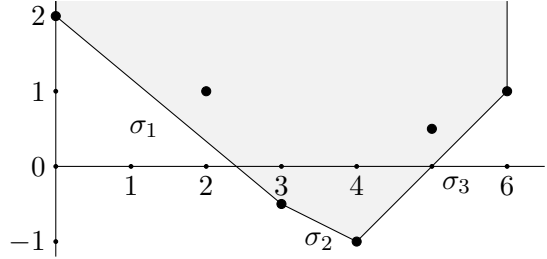
\begin{figure}[htbp]
        \centering
        \begin{tikzpicture}
            \filldraw [color=black!5,draw=black] (0,2.2) -- (0,2) -- (3,-0.5) -- (4,-1) -- (6,1) -- (6,2.2);
            \foreach \x/\y in {0/2,2/1,3/-0.5,4/-1,5/0.5,6/1}
                \filldraw (\x,\y) circle (0.6mm);

            \draw (0,-1.2) -- (0,2.2);
            \foreach \y in {-1,0,1,2}
                \filldraw (0,\y) circle (0.25mm) node [left] {$\y$};
            \draw (0,0) -- (6.5,0);
            \foreach \x in {1,2,3,4,6}
                \filldraw (\x,0) circle (0.25mm) node [below] {$\x$};
            \filldraw (5,0) circle (0.25mm);
            \draw (1.5,0.75) node [below left] {$\sigma_1$};
            \draw (3.5,-0.75) node [below] {$\sigma_2$};
            \draw (5,0) node [below right] {$\sigma_3$};
        \end{tikzpicture}
        \caption{Newton polygon of \[2 + \infty x + 1x^2 + (-\tfrac12) x^3 + (-1)x^4 + \tfrac12x^5 + 1x^6\]}
        \vspace{-1.5em}
        \label{fig:newton-poly-example}
    \end{figure}

    This Newton polygon has three edges: $\sigma_1, \sigma_2, \sigma_3$ with horizontal lengths $3, 1$ and $2$ respectively. The slopes of the edges are respectively $-5/6, -1/2$ and $1$.
\end{example}

\subsection{Initial Forms}
Given a polynomial over $\T$ or $\TR$ we can speak of its \emph{initial form} which is roughly the sum of all monomials corresponding to the points contained inside a single edge of its Newton polygon.

Given a polynomial $p = \sum c_i x^i \in \T[x]$ and $a \in \T$. Let $C$ be the minimum of all the extended real numbers $c_i + a i \in \R \cup \{\infty\}$. Define
\[ \In_a(p) = \sum\{x^i : c_i + ai = C \} \in \K[x]. \]
If $\sigma$ is an edge of $\newt(p)$ with slope $-a$ then $i \in \supp(\In_a(p))$ if and only if $(i,c_i) \in \sigma$.

Likewise, for a polynomial $p = \sum c_i x^i \in \TR[x]$ and $a \in \T$. Let $C$ be the minimum of $|c_i| + a i$ and define
\[ \In_a(p) = \sum\{\sign(c_i)x^i : |c_i| + ai = C \} \in \Si[x]. \]

If $\sigma$ is an edge of $\newt(p)$ with slope $-a$ then we will define $\In_{\sigma}(p) = \In_a(p)$.

\subsection{Sign Changes}
For a polynomial $p$ over $\Si$ or over any ordered field, we define $\Delta(p)$ to be the number of sign changes in the coefficients of $p$, ignoring $0$s.

\begin{example}
    If $p = 1 + x + 0x^2 + 0x^3 - x^4 + 0x^5 - x^6 + x^7 \in \Si[x]$, then $\Delta(p) = 2$: one sign change from $x$ to $-x^4$, one from $-x^6$ to $x^7$.
\end{example}

For polynomials over $\TR$, what we are interested in are the number of sign changes along an edge $\sigma$ of $\newt(p)$. That is to say $\Delta(\In_{\sigma}(p))$.

If $a$ is a root of $p$ and $\sigma$ is the edge of $\newt(p)$ corresponding to $a$ (meaning the slope is $-a$), we define $\Delta_a(p) = \Delta_\sigma(p) \coloneqq \Delta(\In_{-a}(p))$. If $a$ is not a root of $p$, then we define $\Delta_a(p) = \Delta_{\sigma}(p) = 0$. Of course, if $a$ is not a root, then there is no edge $\sigma$ so we will avoid writing $\Delta_\sigma$ in that case.

\section{The Classical Multiplicity Formula}
In this section we prove Theorem~\ref{thm:classical}. The purpose of this section is to establish a multiplicity formula for formally real ordered fields without using any of the language about hyperrings.

In principle, this allows us to prove the corresponding theorem for $\TR$ (Theorem~\ref{thm:main}) by simply lifting the polynomial to, for example, $\R[[t^\R]]$ as in Theorem~\ref{thm:lifting} and then showing that the multiplicity is preserved under the map $v_\R \colon \R[[t^\R]] \to \TR$ (e.g. Lemma~\ref{lem:morphism-factorization}, \cite[Proposition~B]{BL}). However, we will give our proof of Theorem~\ref{thm:main} entirely within the context of polynomials over hyperrings to show that the theory is self-contained.

\setcounter{maintheorem}{1}
\begin{maintheorem} \label{thm:classical}
    Let $K$ be a real-closed, non-Archimedean-ordered valued field with residue field $k$. Let $P(x) = \sum_{i = 0}^n c_ix^i \in K[x]$ be a polynomial. Suppose the Newton polygon of $P$ has edges $\sigma_1,\dots,\sigma_r$ with slopes $\lambda_1,\dots,\lambda_r$ respectively.
    Let $P_j = \In_{\sigma_j}(P)$.

    Factor $P$ as $P = c_nQ(x)\prod_{i = 1}^l(x - \alpha_i)$ where $Q$ is irreducible, $\deg Q > 1$ and $\alpha_1, \dots, \alpha_l$ are the roots of $P$ in $K$. Then
    \begin{enumerate}
        \item for all $i$, there exists some $j$ such that $v(\alpha_i) = -\lambda_j$
        \item $\Delta(P_j) \ge \#\{i : v(\alpha_i) = -\lambda_j, \alpha_i > 0\}$ and the two quantities differ by an even number.
    \end{enumerate}
\end{maintheorem}

\begin{proof}
    Statement (1) is a result of the classical Newton polygon rule and the reader should be able to find a proof in most sources on valued fields or $p$-adic analysis. For example, Fernando Gouvêa's book \emph{p-adic Numbers} contains a proof in Chapter 7.4\footnote{Chapter 6.5 in the 1st ed.} \cite{Go}.

    For statement (2), let us begin by introducing some notation. For any $\gamma \in v(K^\times)$, we will let $t^\gamma$ denote a fixed element of $K$ such that $v(t^\gamma) = \gamma$ and $t^\gamma > 0$.

    Fix $j \in \{1,\dots,r\}$. Then the polynomial $P(t^{\lambda_j} x)$ has roots $\alpha_i t^{-\lambda_j}$ and the slopes of $\newt(P(t^{\lambda_j} x))$ are $\{\lambda_j - \lambda_i : i = 1, \dots, r \}$.
    In particular, the edge corresponding to $\sigma_j$ now has slope $0$ after this transformation and hence $\gamma \coloneqq v(c_i)$ is constant for each monomial $c_ix^i$ of $P_j(t^{\lambda_j}x)$.

    Consider the polynomial $R(x) \coloneqq t^{-\gamma}P(t^{\lambda_j}x)$. The edge corresponding to $\sigma_j$ in $\newt(R)$ now has slope $0$ and lies along the $x$-axis. In particular, $R \in \mathcal{O}_K[x]$. We note that the number of real roots of $R$ with valuation $0$ is exactly the number of real roots of $P$ with valuation $\lambda_j$ and the signs of the roots match as well.

    The residue of $R$ mod $\mathfrak{m}_K$ has the same support as $P_j$, namely
    \[ \overline{R}(x) = \sum\{\overline{a_it^{i\lambda_j-\gamma}}x^i : (i, v(a_i)) \in \sigma_j\}. \]
    Now factor $\overline{R}$ in $k[\sqrt{-1}]$ as $cx^m\prod_{i = 1}^s(x - \beta_s)$. The roots $\beta_1,\dots,\beta_s$ are the residues of the complex roots of $R$ with valuation $0$. These include residues of real roots plus pairs of residues of complex-conjugate non-real roots. For example, we might have a pair of complex-conjugate roots in $K[\sqrt{-1}]$, whose residue in $k[\sqrt{-1}]$ is real. They will have the same residue since the valuation on $K[\sqrt{-1}]$ is Galois invariant.

    Therefore, the number of positive, real roots of $P$ with valuation $\lambda_j$ is bounded above by, and congruent mod $2$ to the number of positive, real roots of $\overline{R}$. By Descartes's rule, this is counted by the number of sign changes of $P_j$.
\end{proof}

\begin{remark} \label{rem:hyp-counterexample}
    Theorem~\ref{thm:classical} can fail if the ordering is incompatible with the valuation. For example, in $\Q$ with the $p$-adic valuation and the usual (Archimedean) order, the polynomial $-p + (p-1)x + x^2 = (x + p)(x - 1)$ does not have a positive root with valuation $1$.
\end{remark}

\begin{corollary}
    Over $K = \R[[t^\R]]$, the number of sign changes of $P_j$ is also congruent, mod $2$, to the number of complex roots of $P$ with valuation $\lambda_j$, whose leading term is real and positive.
\end{corollary}
The difference between this corollary and the general statement is that elements of $\R[[t^\R]]$ have terms and we can either ask that all of the terms be real or just the leading one.

\begin{proof}
    From the proof of Theorem~\ref{thm:classical}, the non-real roots came in pairs. One of those sets of pairs is the pairs of roots whose leading term is real but whose higher-order terms are non-real. When we add these pairs back into the count, we obtain this corollary.
\end{proof}

\section{The Baker-Lorscheid Multiplicity Formula}
In this section we prove the main theorem, Theorem~\ref{thm:main}. To do this, we first need to understand better what it means to be a root of a polynomial over $\TR$. For instance, over the tropical hyperfield, it is known that $a$ is a root of $\sum c_ix^i$ if and only if the minimum $\min\{ c_i + ia \}$ is a achieved at least twice.

At the end of the section, we will relate Theorem~\ref{thm:main} to Theorem~\ref{thm:classical}.

\begin{proposition}
    A sum $a_1 \boxplus \cdots \boxplus a_n$ in $\TR$ contains $\infty$ if and only if
    \begin{enumerate}
        \item The minimum of $|a_1|, \dots, |a_n|$ is achieved twice, and
        \item the minimum is achieved with opposite signs. That is, there exists $i$ and $j$ such that $a_i = -a_j$ and $|a_i| = |a_j| = \min \{|a_1|, \dots, |a_n|\}$.
    \end{enumerate}
\end{proposition}

\begin{proof}
    If $\infty \in a_1 \boxplus \cdots \boxplus a_n$ then $\infty \in |a_1| \boxplus \dots \boxplus |a_n|$. In analogy with valued fields, if the minimum of a sum over $\T$ occurs just once, then the whole sum evaluates to that minimum. For instance, suppose $|a_1|$ is the unique minimum. Then $|a_1| \boxplus |a_2| = \{|a_1|\}$ and $\{|a_1|\} \boxplus |a_3| = \{|a_1|\}$ and so on. Therefore, for $\infty$ to be in this sum, the minimum has to occur at least twice.

    Now let us permute the summands so that $\{ i : |a_i| = \min_j |a_j| \} = \{1,\dots,m\}$. Then by the same reasoning as above, $a_1 \boxplus \dots \boxplus a_n = a_1 \boxplus \dots \boxplus a_m$. Now applying the $\sign$ morphism, we get $0 \in \sign(a_1) \boxplus \dots \boxplus \sign(a_m)$ which is only possible if there is a pair of signs which is different.

    Finally, that the minimum of $|a_1|, \dots, |a_n|$ is $|a_1| = |a_2|$ and $a_1 = -a_2$. Then one checks that $a_1 \boxplus \dots \boxplus a_n = a_1 \boxplus a_2 \ni \infty$.
\end{proof}

\begin{corollary}
    If $p(x) = \sum_{i = 0}^n c_ix^i$ is a polynomial over $\TR$, then a nonzero $a$ is a root of $p$ if and only if
    \begin{enumerate}
        \item $\min\{ |c_ia^i| \}$ is achieved twice.
        \item $\min\{ |c_ia^i| \}$ is achieved with opposite signs.
    \end{enumerate}
\end{corollary}

\setcounter{maintheorem}{0}
\begin{maintheorem} \label{thm:main}
    Let $p = \sum_{i = 0}^n c_i x^i \in \TR[x]$. Let $a$ be a positive root of $p$. Then $\mult_a(p) = \Delta_a(p)$.
\end{maintheorem}

\begin{proof}
    The goal, by definition, is to maximize $\mult_a(q)$ over all factorizations $p \in (x - a)q$.
    That is, we claim:
    \begin{enumerate}
        \item If $p \in (x - a)q$ is any factorization, then $\Delta_a(q) < \Delta_a(p)$. \label{claim:1}
        \item There exists a factorization $p \in (x - a)q$ such that $\Delta_a(q) = \Delta_a(p) - 1$. \label{claim:2}
    \end{enumerate}
    The Theorem follows from these two claims by induction.

    The proof of Claim \eqref{claim:1} is identical to the corresponding proof for polynomials over $\Si$ \cite[proof of Theorem~3.1]{BL}. To make this connection clear, we apply the same transformations as in the proof of our Theorem~\ref{thm:classical}. That is: first consider $p(ax)$ and then let $r(x) = \gamma^{-1} p(ax)$ where $\gamma \coloneqq |c_i|$ is constant for each monomial $c_ix^i$ of $\In_{\sigma}(p)(ax)$.

    Then, as before, $\newt(r)$ lies above the $x$-axis with the edge corresponding to $\sigma$ now having a slope of $0$ and lying on the $x$-axis. So when we take the ``residue,'' $r|_{t = 0} \in \Si[x]$, we are recording just the signs of the monomials in $\In_{\sigma}(p)$.

    To summarize, the number of sign changes does not change when we do our transformations
    \[\TR[x] \xrightarrow{p \mapsto \gamma^{-1}p(ax)} \OTR[x] \xrightarrow{\cdot|_{t = 0}} \Si[x]. \]

    Now consider a factorization $p \in (x - a)q$ in $\TR[x]$ and the corresponding factorization $r|_{t = 0} \in (x - 1)s|_{t = 0}$ in $\Si[x]$, where $s(x) = \gamma^{-1}q(ax)$. Passing the factorizations through these morphisms is justified by Lemma~\ref{lem:morphism-factorization}. Therefore, applying Baker and Lorscheid's proof to this factorization, we have \[ \Delta_\sigma(p) = \Delta(r|_{t = 0}) > \Delta(s|_{t = 0}) = \Delta_\sigma(q). \]

    For claim \eqref{claim:2}, it will be convenient to replace $p(x) \in \TR[x]$ by its transformation $r(x)  \in \OTR[x]$. That is, we may assume that $a = (+1,0)$ and that $\sigma$ has slope $0$ and lies on the $x$-axis and that $\newt(p)$ lies on or above the $x$-axis.

    We note that our transformation $\TR[x] \to \OTR[x]$ is invertible so there is no issue pulling a factorization back along it. On the other hand, the transformation $\OTR[x] \to \Si[x]$ is not invertible so rather than try to lift a factorization along this morphism, we will work inside $\OTR[x]$ directly.

    Let us fix our notation. Write
    \[ p = c_0 + c_1x + \dots + c_nx^n \]
    and suppose the endpoints of our edge $\sigma$ are $(n_1, 0)$ and $(n_2, 0)$. Suppose that the first sign change in $\sigma$ occurs at index $r$. That is, $c_i = c_{n_1}$ or $|c_i| > |c_{n_1}|$ for $i = n_1,\dots,r-1$ and $c_r = -c_{n_1}$.

    We will define a polynomial $q = \sum d_ix^i$ as follows.
    \begin{itemize}
        \item For $i \le n_1$, define $d_i = -c_k$ where $k$ is the smallest index such that $|c_k| = \min_{j \le i} |c_j|$.
        \item For $n_1 \le i \le r$, let $d_i = -c_{n_1}$.
        \item For $r \le i \le n$, let $d_i = c_k$ where $k$ is the smallest index $> i$ such that $|c_k| = \min\{|c_j| : j > i \}$.
    \end{itemize}
    By construction, $\Delta_a(q) = \Delta_a(p) - 1$ because we have all the sign changes of $p$ except for the one at $i = r$. All that is left to do is to show that $p \in (x - a)q$.

    For $i = 0$ and $i = n$, we have $d_0 = -c_0$ and $d_{n-1} = c_n$ which confirms the identity $p \in (x - a)q$ on the ``boundary.'' For all other $i$, the relation $p \in (x - a)q$ means that
    \begin{equation*}
        c_i \in d_{i - 1} \boxplus (-d_i).
    \end{equation*}

    For $i \le n_1$, there are two cases. First, if $|c_i| < \min_{j < i} |c_j| = |d_{i - 1}|$ then $d_i = -c_i$, in which case $d_{i - 1} \boxplus (-d_i) = \{-d_i\} = \{c_i\}$. Second, if $|c_i| \ge \min_{j < i} |c_j|$ then $d_i = d_{i - 1}$ and $d_{i - 1} \boxplus (-d_i) \ni c_i$.

    Next, for $n_1 < i \le r$, $d_i$ is constant and hence $d_{i-1} \boxplus (-d_i) \ni c_i$ for any $i$ in that range.

    Finally, for $r < i < n$, notice that $|c_i| \ge |d_{i - 1}|$. If \[ |d_i| = \min\{|c_k| : k > i\} > |d_{i-1}| = \min\{ |c_k| : k > i - 1\},\] then we must have $d_{i - 1} = c_i$ and hence $c_i \in d_{i-1} \boxplus (-d_i)$. Otherwise, if $d_{i - 1} = d_i$ then $c_i \in d_{i-1} \boxplus (-d_i)$ since $|c_i| \ge |d_{i-1}|$. The final case, which is that $d_{i - 1} = -d_i$ can only happen if $d_{i-1} = c_i$ and again $c_i \in d_{i-1} \boxplus (-d_i)$.
\end{proof}

\begin{remark}
    For a polynomial in $\Si[x] \subseteq \TR[x]$, the polynomial $q$ we define is the same as the corresponding $q$ defined in Baker and Lorscheid's proof \cite[Theorem~3.1]{BL}. Also notice that $\newt(q)$ is the same as $\newt(p)$ except for shortening $\sigma$ by $1$.
\end{remark}

The following Theorem is an analogue of Remark 1.14 in Baker and Lorscheid's paper~\cite{BL} concerning lifting polynomials in $\TR[x]$ to polynomials in $\R[[t^\R]][x]$ with the maximum number of real roots as allowed by the Theorem~\ref{thm:main}.

\begin{theorem} \label{thm:lifting}
    Let $p \in \TR[x]$, then there exists a polynomial $P \in \R[[t^\R]][x]$ such that $v_\R(P) = p$ and for each edge $\sigma$ of $\newt(P)$ with slope $-|a|$ ($a \in \TR, a > 0$), $P$ has exactly $\mult_{a}(p(x))$ (resp. $\mult_{-a}(p(-x))$) roots of valuation $|a|$ whose leading term is real and positive (resp. negative).
\end{theorem}
\begin{proof}
    For each $\sigma$, let $r_\sigma(x)$ be the polynomial $\gamma^{-1}p(ax)$ that we had in the proof of the Theorem~\ref{thm:main}. Now we use the converse of Descartes's rule (for example, \cite[Theorem~1]{G}) to choose a lift $R_\sigma(x) \in \R[x]$ of $r_\sigma|_{t = 0} \in \Si[x]$ which has exactly $\mult_a(p(x))$ (resp. $\mult_{-a}(p(-x))$) positive (resp. negative) real roots.

    Now let \[ \tilde{P}(x) = t^{\delta}\prod_\sigma R_\sigma(t^{-|a|} x) \] where $\delta$ is chosen such that $\newt(\tilde{P}) = \newt(p)$. From $\tilde{P}$, we can take any choice of monomials not on the boundary of $\newt(p)$ to a polynomial $P$ such that $v_\R(P) = p$.

    Note that the leading terms of the roots of $\In_{\sigma}(P)$ are exactly the roots of $R_\sigma \in \R[x]$. By construction, $P$ has correct number of roots of valuation $|a|$ whose leading term is real and positive/negative.
\end{proof}

It would be interesting to know if we could pick $P$ such that the roots themselves are real, rather than just their leading terms. Also note that this theorem only ensures the maximal number of roots with one specific valuation, rather than all possible valuations at once.

\emergencystretch=1em % fix hbox issues in bibliography
\printbibliography

% \bigskip

% \noindent \emph{Email address:} \texttt{\href{mailto:tgunn@gatech.edu}{tgunn@gatech.edu}}

% \bigskip

% \noindent \textsc{\footnotesize School of Mathematics, Georgia Institute of Technology, Atlanta, USA}

\end{document}